\documentclass[a4paper,12pt]{article}

\usepackage{amsthm}
\usepackage{amsmath}
\usepackage{amssymb}

\numberwithin{equation}{section}

\newtheorem{thm}{Theorem}[section]
\newtheorem{prop}[thm]{Proposition}
\newtheorem{lem}[thm]{Lemma}
\newtheorem{dfn}[thm]{Definition}

\def\ds{\displaystyle}
\def\ol{\overline}
\def\R{\mathbb{R}}
\def\Ri{\mathbb{R}\cup\{+\infty\}}
\def\N{\mathbb{N}}
\def\eps{\varepsilon}

\def\span{\mathrm{span\,}}

\def\argmin{\mathrm{argmin\,}}

\def\be{\begin{equation}}
\def\ee{\end{equation}}

\def\ba{\begin{array}}
\def\ea{\end{array}}

\usepackage{amsmath}

\begin{document}

\title{Perturbation Method in Orlicz Sequence Spaces}
\author{H. Topalova\thanks{The author is supported by the Bulgarian Ministry of Education and Science under the National Program ``Young Scientists and Postdoctoral Students 2''.} \and N. Zlateva\thanks{The author is supported by the Scientific Fund of Sofia University.}}

\date{\today}
\maketitle

\begin{abstract}
  We develop a new perturbation method in Orlicz sequence spaces $\ell_M$ with  Orlicz function $M$ satisfying $\Delta_2$ condition at zero. This result allows one to support from below any bounded below lower semicontinuous function with bounded support, with a perturbation of the defining function $\sigma_M$.

  We give few examples how the method can be used for determining the type of the smoothness of certain Orlicz spaces.


\end{abstract}

\section{Introduction}
Variational principles are well established and widely used tools in renorming theory of Banach spaces and especially so after the prominent monograph~\cite{DGZ}.

In this article we establish a new perturbation tool -- or variational principle -- fit to the structure of Orlicz sequence spaces. If the Orlicz function $M$ satisfies $\Delta_2$ condition at zero then for each proper lower semicontinuous and bounded below function $f$ on the Orlicz space $\ell_M$ and each $\varepsilon>0$  there is a perturbation
$$
  g_a(x) = \sum_{n=1}^\infty a_nM(x_n)
$$
of the defining function $\sigma_M(x)=\sum_{n=1}^\infty M(x_n)$ such that $0\le a_n < \varepsilon$ for all $n\in\N$ and $f+g_a$ attains a minimum, see for details Theorem~\ref{thm:perturb-general}.

This result is proved by adapting some ideas from \cite{iv-zla-jca,iv-zla-jota}.

The benefit of such perturbation, i.e. a perturbation keeping the form of $\sigma_M$, is that we can adapt the proof of   \cite[Lemma II.5.4]{DGZ}, see also \cite{Stankov}, to finer situations. That is, we apply the above result for the function $f-\sigma_M$, where $f$ is a proper lower semicontinuous and bounded below function with bounded domain (and, therefore, $f-\sigma_M$ is bounded below) to get the following statement, see for details Theorem~\ref{thm:perturb-strict}.

If $f$ is a proper lower semicontinuous and bounded below function with bounded domain then there are $a_n\in[1,2]$ such that the function
$$
  x \to f(x) -  \sum_{n=1}^\infty a_nM(x_n)
$$
attains a minimum.

In other words, we can support $f$ from below with a function very similar to $\sigma_M$. So, if $M$, and thus $\sigma_M$, is ``too rough upwards'' then $f$ can not be ``too smooth'' inside its domain. This idea applied to $f=b^{-2}$, where  $b$ is a nontrivial bump on $\ell_M$, may prove the non-existence of a bump with certain smoothness properties.

 When applied to $\ell_1$ the latter gives that there is no nontrivial Fr\'echet differentiable bump on $\ell_1$, see Proposition~\ref{prop:l1}. We can use this simple example to illustrate an important point. Another way to prove the nonexistence of nontrivial \emph{Lipschitz} Fr\'echet differentiable bump on $\ell_1$ is to use as a perturbation space the Banach space of all bounded Lipschitz continuous and  Fr\'echet differentiable functions and apply \cite[Lemma I.2.5]{DGZ}. But for doing  this one needs global Lipschitz continuity of the bump. Therefore, it is more beneficial  -- here and in general -- to perturb $\sigma_M$, which in this case is the norm. (Note that the proof works for any Orlicz function $M$ satisfying $\limsup_{t\to 0}M(t)/t >0$, but from convexity it follows that such function is equivalent at zero to $|\cdot|$, so the space $\ell_M$ is isomorph to $\ell_1$.)

In a similar simple way we prove in Proposition~\ref{prop:lp} that if  $M$ has $\Delta_2$ at zero and $\limsup_{t\to 0}M(t)/t^p =\infty$ for some $p\in(1,2]$  then on $\ell_M$ there is no nontrivial bump $b$ with the estimate $b(x+h) = b(x) + \langle b'(x),h\rangle + O(\|h\| ^p)$ at each point $x$. In particular, on $\ell_p$, $p\in(1,2]$  there is no nontrivial bump $b$ with the estimate   $b(x+h) = b(x) + \langle b'(x),h\rangle + O(\|h\| ^q)$ for $q>p$. In  \cite[Chapter~V]{DGZ} it is proved  that there is no nontrivial bump with locally $(q{-}1)$-H\"older derivative, so our result apparently is bit more general. Our approach is also more direct, because in \cite{DGZ} a deep result is used to regularise the bump to one which has globally $(q{-}1)$-H\"older derivative.

Our next example, Proposition~\ref{prop:M''}, is somewhat different. From the previous example it immediately follows that on $l_p$ for $p\in (1,2)$ there is no twice Fre\'echet differentiable bump, because two times Fre\'echet differentiability implies the estimate $b(x+h) = b(x) + \langle b'(x),h\rangle + O(\|h\| ^2)$. It is known, see \cite{her-tro}, but not easy to prove that there is even no   twice G\^ateaux differentiable bump. We show that if $M$ has $\Delta_2$ at zero and $\lim _{t\to 0}M''(t)=\infty $, then  there is no twice G\^ateaux differentiable bump on $l_M$. In particular, there is no twice G\^ateaux differentiable bump on $l_p$ for $p\in (1,2)$.

These applications can be proved by Stegall Variational Principle (\cite{stegal1,stegal2} and \cite[Theorem 11.6]{FHHMZ}) as for example in \cite{MNZ}, but the approach here is more straightforward and explicit.

The article is organised as follows. In Section~\ref{sec:abstract} we define the perturbation space, see Definition~\ref{pert_sp} and use it to prove an abstract variational principle in Theorem~\ref{min}. Even more abstract variant is given in Theorem~\ref{thm:min-abstract}. In Section~\ref{sec:Orlicz} we give some facts about Orlicz sequence spaces we need in the sequel and prove Lemma~\ref{lem:sigma-well-posed} which is of independent interest. In Section~\ref{sec:main} we find a perturbation space in Orlicz sequence space where the Orlicz function $M$ is non-degenerate and satisfies $\Delta_2$ condition at $0$, see  Lemma~\ref{petsp}. Theorem~\ref{thm:perturb-general} is a variational principle in such a Orlicz space. The latter is used for proving Theorem~\ref{thm:perturb-strict} which is our tool for  determining the type of the smoothness of certain Orlicz spaces. Such results are presented in the last Section~\ref{sec:appl}.

\section{Abstract perturbation technique}\label{sec:abstract}

We need to refine the main technique of \cite{iv-zla-jca}, because the axioms as formulated there will not do in our context here. We use the opportunity to also somewhat simplify the proof. We start with recalling some definitions.

Let $(X, \| . \| )$ be a Banach space and let $S$ be a nonempty
subset in $X$.
By $\alpha(S)$  we denote the Kuratowski index of non-compactness of $S$ (see \cite{K}), that is the infimum of all $\eps >0$ for which $S$
admits finite (or compact) $\eps$ net.

For a function $f:X\to \Ri$ and $\eps\ge 0$  we set
\[
\Omega _f^S(\eps)=\left\{ x\in S: f(x)\le \inf _S f+\eps\right\}.
\]
 Clearly, $\Omega_{f}^S(0)=\argmin_S f$.

 We will need  the following simpler version of \cite[Lemma 4]{iv-zla-jca}: for any $f,g:X\to \Ri$, bounded below on $S$, and $\delta >0$,
\begin{equation}
  \label{eq:new-lema4}
  \Omega_f^S(\delta)\cap \Omega_g^S(\delta)\ne\varnothing\quad\Longrightarrow\quad
  \Omega_{f+g}^S(\delta)\subset\Omega_f^S(3\delta)\cap \Omega_g^S(3\delta).
\end{equation}
Indeed, let $x\in\Omega_f^S(\delta)\cap \Omega_g^S(\delta)$ and $y\in \Omega_{f+g}^S(\delta)$. Then
$  \delta \ge f(y)+g(y) - \inf(f+g)\ge f(y)+g(y) - (f(x)+g(x)) \ge (f(y)-\inf_Sf-\delta) + (g(y)-\inf_Sg-\delta)$. Since $g(y)-\inf_Sg\ge0$, we derive $\delta\ge f(y)-\inf_Sf - 2\delta$. So, $y\in \Omega_f^S(3\delta)$. In the same way, $y\in \Omega_g^S(3\delta)$.

 \begin{dfn}[\cite{iv-zla-jota}, Definition 4.1]\label{wpmc}
  Let $S\subset X$ and $f:S\to\R\cup\{+\infty\}$. We say that the minimization
  problem
  \begin{equation}
    \label{eq:problem}
    (f,S) \quad \ \left\{\ba{l}
    f(x)\to\min\\
    x\in S\ea \right.
  \end{equation}
  is \emph{well-posed modulus compact}, abbreviated as \emph{wpmc}, if $\inf_S f\in\R$ and
  $\Omega_f^S(0)$ is nonempty and compact set such that if
  $\{x_{n}\}_{n=1}^{\infty}$ is a minimizing sequence:
  \[
  f(x_{n})\to\inf_S f,\quad n\to\infty,\]
  then\[
  d(x_{n},\Omega_f^S(0))\to 0.\]
  \end{dfn}

  The meaning of the term wpmc is quite natural. If $\Omega_f^S(0)$ is a singleton, then   we get the well-known notion of Tikhonov well-posedness (for details
  see e.g. \cite[Chapter I]{zolezzi_book}).
   From Kuratowski Lemma we have, see \cite[Proposition 4.2]{iv-zla-jota}:
  for a closed set $S\subset X$ and a proper lower semicontinuous function $f:S\to\R\cup\{+\infty\}$ bounded below on $S$
  \begin{equation}
    \label{eq:wpmc-characterization}
    (f,S) \text{ is wpmc }\iff \lim_{t\downarrow 0}\alpha(\Omega_{f}^S(t))=0.
  \end{equation}
  Clearly, if $X$ is finite dimensional and $S$ is closed then for any proper lower semicontinuous function $f:S\to\R\cup\{+\infty\}$ bounded below on $S$ the problem $(f,S)$ is wpmc. The idea for general $X$ is to perturb $f$ by a suitable function so that the perturbed problem is wpmc. Here we differ from the usual variational principles, as those considered in \cite{DGZ}, where the problem is perturbed to become Tikhonov well posed.

\begin{dfn}\label{pert_sp}
Let $(X,\|\cdot\|)$ be a Banach space and let $S\subseteq X$ be a non-empty  set. The space of
$(\cal{P},\|\cdot\|_{\cal{P}})$ of real continuous and bounded  on $S$ functions,  is called a \emph{perturbation space on} $S$ if
\begin{itemize}
\item[$\mathrm{(\i)}$] $\cal{P}$ is a convex cone;
\item[$\mathrm{(\i\i)}$]  $\cal{P}$ is complete with respect to the norm $\|\cdot\|_{\cal{P}}$ defined on $\span \cal{P}$,  that dominates the uniform  convergence on $S$, that is, for some $c>0$
\begin{equation}
  \label{eq:dominate-uniform}
  \sup_{x\in S} |g(x)| \le c \|g\|_{\cal{P}},\quad\forall g\in \span \cal{P}.
\end{equation}
In other words, $(\span \cal{P},\|\cdot\|_{\cal{P}})$ is a Banach space of bounded on $S$ continuous functions, and $\cal{P}$ is its positive cone.
\item[$\mathrm{(\i\i\i)}$]  For any $\eps >0$ there exists $\delta>0$    such that for any $x\in S$ there exists $g\in \cal{P}$ (depending on $x$) such that
\begin{equation}
  \label{eq:def:perturb:c}
  \|g\|_{\cal{P}}\le \varepsilon,\ x \in \Omega_g^S(\delta), \text{ and }\ \alpha(\Omega_g^S(3\delta)) \le \varepsilon.
\end{equation}
\end{itemize}
\end{dfn}
Note that in the axioms used in \cite{iv-zla-jca} it is required that $x \in \Omega_g^S(0)$, which is too strong.  That is why we weaken this requirement.

Here is the promised refinement of \cite[Theorem 10]{iv-zla-jca}.
\begin{thm}\label{min}
  Let $(X,\|\cdot\|)$ be a Banach space and let $S\subset X$ be a non-empty closed set in $X$. Let  $(\cal{P},\|\cdot\|_{\cal{P}})$ be a perturbation space on $S$. Let   $f:X\to \Ri$ be a  proper, lower semicontinuous and bounded below on $S$ function.

  Then for any $\eps >0$ there exists $g\in \cal{P}$ such that $\|g\|_{\cal{P}} < \eps$ and
  $(f+g,S)$ is wpmc. In particular, $f+g$ attains its minimum on $S$.
\end{thm}
\begin{proof}
 Consider for $n\in\N$ the subset $A_n$ of $\cal{P}$ defined by
\begin{equation}\label{eq:an-def}
 A_n:=\left\{ g\in {\cal{P}}:  \exists t>0  :\ \alpha\left(\Omega_{f+g}^S(t)\right) < \frac{1}{n}\right\}.
\end{equation}
We will show that $A_n$ is dense and open in $(\cal{P},\|\cdot\|_{\cal{P}})$. Then by Baire Category Theorem there will be
$g\in \cap A_n$ such that $\|g\|_{\cal{P}}<\eps$ and by \eqref{eq:wpmc-characterization} the problem  $(f+g,S)$ will be  wpmc.

Fix $n\in\N$. Let $g\in A_n$ be arbitrary and let $\beta > 0$ be such that
$$
  \alpha\left(\Omega_{f+g}^S(3\beta)\right) < \frac{1}{n}.
$$
For any $h\in \span\cal{P}$ such that $\|h\|_{\cal{P}}<\beta/2c$ it follows from \eqref{eq:dominate-uniform} that $h(x) - h(y) < \beta$ for any $x,y\in S$ and, therefore, $\Omega_h^S(\beta) = S$. From \eqref{eq:new-lema4} for $(f+g)$ and $h$ it follows that
$$
  \Omega^S_{f+g+h}(\beta) \subset \Omega_{f+g}^S(3\beta),
$$
thus $\alpha (\Omega^S_{f+g+h}(\beta)) < 1/n$ and $g+h\in A_n$. So, $A_n$ is open.

Let now $h\in \cal{P}$ be arbitrary. Fix an arbitrary $\varepsilon \in (0,1/n)$.
Let $\delta > 0$ be provided by $\mathrm{(\i\i\i)}$ of Definition~\ref{pert_sp}. Fix
$$
  x \in \Omega_{f+h}^S (\delta),
$$
and let $g\in \cal{P}$ satisfy \eqref{eq:def:perturb:c}. The latter implies that $x\in \Omega_g^S(\delta)$. From  \eqref{eq:new-lema4} it follows that
$$
  \Omega_{f+h+g}^S(\delta) \subset \Omega_g^S (3\delta) \Longrightarrow \alpha (\Omega_{f+h+g}^S(\delta)) \le \varepsilon.
$$
As $\varepsilon<1/n$, this means that $h+g\in A_n$. Since $\|g\|_{\cal{P}}\le\varepsilon$, see \eqref{eq:def:perturb:c}, the distance from $h$ to $A_n$ is smaller than $\eps$. In other words, $A_n$ is dense in~$\cal{P}$.
\end{proof}

For the purpose of this article we made the axioms in Definition~\ref{pert_sp} more immediate to check rather than most general. It is worth however to point out how our approach relates to the more general topological context of~\cite{rev-las} where  variational principles are derived from the properties of the $\varepsilon$-argmin map
$$
  g \to \Omega_{f+g} (\varepsilon).
$$
In such context the abstract of our approach is as follows.
\begin{dfn}\label{def:cantor-mes}
  Let $S$ be a non-empty closed subset of the Banach space~$X$. We say that the positive function $\gamma$ defined on the closed subsets of $S$ is \emph{Cantor measure} if
  \begin{itemize}
    \item[$\mathrm{(\i)}$] $\gamma$ is monotone with respect to  inclusion, that is,
    $$
      A\subset B\subset S \Rightarrow \gamma(A) \le \gamma(B).
    $$
    \item[$\mathrm{(\i\i)}$] $\gamma$ satisfies the Cantor Lemma in the sense that if the non-empty subsets $(A_n)_1^\infty$ of $S$ are nested and closed, that is, $A_{n+1}\subset A_n$ for all $n\in\mathbb{N}$, then
    $$
      \lim_{n\to\infty} \gamma(A_n) = 0\ \Longrightarrow\ \bigcap_{n=1}^\infty A_n \neq\varnothing.
    $$
  \end{itemize}
\end{dfn}
Examples of Cantor measures are e.g. the classical one $\mathrm{diam}\, A$, the Kuratovski index of non-compactness $\alpha(A)$ with which we predominantly work here, and the similar De Blasi  index of  non-weak-compactness $\beta (A) = \inf \{ \varepsilon >0:\ A\subset K+\varepsilon B_X,$ where $K\subset X$ is weakly compact set$\}$, see \cite{de-blasi}.

We can call the problem $(f,S)$, see \eqref{eq:problem}, $\gamma$-well-posed if
$$
  \lim_{\varepsilon\to0}\gamma(\Omega^S_f(\varepsilon)) = 0,
$$
where $\gamma$ is a Cantor measure on $S$. Obviously, a $\gamma$-well-posed problem $(f,S)$ has a solution, that is, $\Omega^S_f(0)\neq\varnothing$.

The even more abstract version of Theorem~\ref{min} is:
\begin{thm}\label{thm:min-abstract}
  Let $S$ be a non-empty closed subset of the Banach space $X$ and let $\gamma$ be a Cantor measure on $S$.

  Let $(\mathcal{P},+,\rho)$ be a compete metric semigroup with zero. Let $\mathcal{P}_1 = \mathcal{P}\cup\{f\}$ and let $(\mathcal{P}_1,+)$ be a semigroup. Let the maps
  $$
    T_\varepsilon: \mathcal{P}_1\to 2^S,
  $$
  where $\varepsilon > 0$, satisfy
  \begin{equation}
    \label{eq:min-abs-1}
    \ T_\delta (u) \subset T_\varepsilon (u),\quad \forall u\in \mathcal{P}_1, \ \forall 0<\delta<\varepsilon,
  \end{equation}
  \begin{multline}
    \label{eq:min-abs-2}
    \forall u,v\in \mathcal{P}_1, \ \forall \delta>0: T_{\delta}(u)\cap T_{\delta}(v)\neq\varnothing \Rightarrow\\
    T_{\delta}(u+v) \subset T_{3\delta}(u)\cap T_{3\delta}(v),
  \end{multline}
  \begin{equation}
    \label{eq:min-abs-3}
    \forall\varepsilon>0\ \exists \delta>0:\ \rho(0,g) < \delta \Rightarrow T_\varepsilon (g) = S,
  \end{equation}
  \begin{multline}
    \label{eq:min-abs-4}
    \forall \varepsilon>0\ \exists\delta > 0: \ \forall x\in S\ \exists g\in\mathcal{P}:\\
    \rho(0,g)\le\varepsilon, \ x\in T_{\delta}(g)\text{ and }
    \gamma (T_{3\delta}(g)) \le  \varepsilon.
  \end{multline}
  Then the set $U\subset\mathcal{P}$ defined as
  $$
    g\in U \iff \bigcap_{\varepsilon>0}T_{f+g}(\varepsilon) \neq\varnothing
  $$
  is a dense $G_\delta$ subset of $\mathcal{P}$.
\end{thm}
Clearly, if ${\cal P}$ is a perturbation space, and $f$ is a proper lower semicontinuous and bounded below on $S$ function, Theorem~\ref{thm:min-abstract} implies Theorem~\ref{min}.

Theorem~\ref{thm:min-abstract} can be proved following exactly the proof of Theorem~\ref{min}, that is, by showing that the sets
$$
  A_n := \left\{ g\in {\cal{P}}:  \exists s>0  :\ \gamma\left(T_s\right) < \frac{1}{n}\right\}
$$
are dense and open. For the openness we have to add only that, because the addition is continuous, the set $\{g+h:\ \rho(0,h)<\delta\}$ is a neighbourhood of $g$ in~$\mathcal{P}$.

\section{Some facts about Orlicz sequence spaces}\label{sec:Orlicz}
 Here we recall some well known facts about Orlicz sequence spaces, see e.g. \cite{LT1,HJ}.

An Orlicz function $M$ is a continuous non-decreasing and convex
function defined for $t\ge 0$ such that $M(0)=0$ and $\lim_{t\to \infty} M(t)=\infty$. If $M(t)=0$ for
some $t> 0$, $M$ is said to be a degenerate Orlicz function.

We represent the elements of the space $\ell_\infty$ of bounded sequences as\linebreak
 $x=(x_1,x_2,\ldots,x_n,\ldots)$ or
$$
  x = \sum_{n=1}^\infty x_n e_n,
$$
where $(e_n)_1^\infty$ is the canonical basis of $c_0$ -- the space of tending to zero sequences. Below we will clarify the exact convergence meaning of this sum in our context. We denote  for $x\in\ell_\infty$ and $N\in\mathbb{N}$,
\begin{equation}
  \label{eq:proj:def}
  P_Nx := \sum_{i=1}^N x_ie_i,\quad P_N^\bot x := x - P_Nx = \sum_{i=N+1}^\infty x_ie_i.
\end{equation}

The even convex function $\sigma_M:\ell_\infty \to [0,\infty]$ is defined
 by
 \begin{equation}
  \label{eq:def-sigma}
  \sigma_M(x):= \sum_{n=1}^\infty  M(|x_n|).
 \end{equation}

The Orlicz sequence space $\ell_M$ is the linear subspace of $\ell_\infty$ consisting of all $x\in \ell_\infty$ satisfying
 $\sigma_M(x/\rho)<\infty$ for some $\rho>0$. Equipped with the norm given by the Minkowski
functional of $\{x\in \ell_\infty:\sigma_M(x)\le 1\}$, i.e.
\[
\|x\|:=\inf \left\{ \rho>0: \sigma_M\left(\frac{x}{\rho}\right)\le 1\right\},
\]
the space $\ell_M$ is a Banach space. If $M$ is degenerate then $\ell_M \equiv \ell_\infty$ and this case is not interesting for us. If $M$ is non-degenerate then, of course, $\ell_M\subset c_0$.

From the
convexity of $\sigma_M$ we have
\be\label{F}
\sigma_M(x)\le \|x\|,\text{ if }\|x\|\le 1;\quad \sigma_M(x) > \|x\|,\text{ if }\|x\| > 1.
\ee
The non-degenerate Orlicz function $M$  \emph{satisfies the $\Delta_2$ condition at zero} if
\begin{equation}
  \label{eq:def-delta-2-limsup}
 \liminf_{t\downarrow 0}\frac{M(t)}{M(2t)} > 0.
\end{equation}
By the end of this section we work with  a non-degenerate Orlicz function $M$ satisfying the $\Delta_2$ condition at zero.

Fix $\ol t$ such that
\begin{equation}
  \label{eq:ove-t}
  \ol t > 1, \text{ and }M(\ol t) > 1.
\end{equation}
Since the function $M(t)/M(2t)$ is continuous on $(0,\ol t]$, from \eqref{eq:ove-t} it follows that $\ds \inf_{0< t \le \ol t}M(t)/M(2t) > 0$ and, therefore, there is $C \in (0,1)$ such that
\begin{equation}
  \label{eq:delta-2}
M(t)\ge CM(2t),\quad \forall t\in [0,\ol t].
\end{equation}

Note that from \eqref{eq:def-sigma}, \eqref{eq:ove-t}, \eqref{eq:delta-2} and $B_{\ell_M} = \{x:\ \sigma_M(x)\le 1\}$, see \eqref{F},  it immediately follows that
  \begin{equation}
    \label{eq:sigma-delta-2}
    \sigma_M(x) \ge C \sigma_M(2x),\quad \forall x\in B_{\ell_M}.
  \end{equation}

 Let $m\in\mathbb{N}$ and let $x\in\ell_M$ be  such that $\|x\|\le 2^m$, that is, $\sigma_M(2^{-m}x) \le 1$. In particular, $M(2^{-m}|x_n|) \le 1$ for all $n$, so \eqref{eq:ove-t} gives $|x_n|\le 2^m\ol t$ for all $n\in\mathbb{N}$.
  Let $A\subset\mathbb{N}$ be the set of those $n$ for which $|x_n|\ge \ol t$. Because $x\in c_0$ the set $A$ is finite. Let $|A|$ be the number of elements of $A$. Clearly, $1 \ge \sigma_M(2^{-m}x) \ge \sum_{n\in A} M(2^{-m}|x_n|)\ge |A|M(2^{-m}\ol t)$.
  So, $|A|\le 1/M(2^{-m}\ol t)$, and $\sum_{n\in A} M(|x_n|) \le M(2^m\ol t)/M(2^{-m}\ol t)$.
  On the other hand, \eqref{eq:delta-2} iterated gives $1 \ge \sigma_M(2^{-m}x) \ge \sum_{n\not\in A} M(2^{-m}|x_n|) \ge C^m\sum_{n\not\in A} M(|x_n|)$.  So, $\sum_{n\not\in A} M(|x_n|)\le C^{-m}$ and $\sigma_M(x)\le C^{-m}+M(2^m\ol t)/M(2^{-m}\ol t)$ for all $x\in 2^m B_{\ell_M}$.
  Consequently,
  \begin{equation}
    \label{eq:sigma-finite}
    \sup _{KB_{\ell_M}}\sigma_M < \infty,\quad\forall K>0.
  \end{equation}
 Also, provided $M$ is non-degenerate and satisfying the $\Delta_2$ condition at zero, $\sigma_M$ is convex and bounded on the bounded sets and, therefore, continuous  (see e.g. Lemma~\ref{lem:Lips}). In this case, \eqref{F} implies  that
  \begin{equation}
    \label{eq:sphere}
    \|x\| = 1 \iff \sigma_M(x) = 1.
  \end{equation}

We need the $\Delta_2$ condition mostly because  of the following lemma.
\begin{lem}\label{lem:sigma-well-posed}
  Let $M$ be a non-degenerate Orlicz function. Then $M$ satisfies the $\Delta_2$ condition at zero
  if and only if $\sigma_M$ has a strong minimum at zero, or, in other words, the problem
  $$
    \sigma_M(x) \to \min,\quad x\in\ell_M
  $$
  is Tikhonov well posed.
 \end{lem}
\begin{proof}
  Let $M$ be a non-degenerate Orlicz function that  satisfies the $\Delta_2$ condition at zero. Iterating on \eqref{eq:sigma-delta-2} we get
  $$
    \sigma_M\left(\frac{x}{2^m}\right) \ge C\sigma_M\left(\frac{x}{2^{m-1}}\right) \ge\ldots\ge C^m\sigma_M(x),\quad\forall x\in B_{\ell_M},\ \forall m\in\mathbb{N}.
  $$
  So, if $y\in\ell_M$ is such that $\|y\| > 2^{-m}$, then using \eqref{eq:sphere},
  $$
    \sigma_M(y) > \sigma_M \left(\frac{y}{2^m\|y\|}\right)\ge C^m\sigma_M\left(\frac{y}{\|y\|}\right)=C^m.
  $$
Therefore,
$$
  \Omega_{\sigma_M}(C^m) \subset 2^{-m}B_{\ell_M},
$$
and
\begin{equation}
  \label{eq:diam-omega-0}
  \mathrm{diam}\, \Omega_{\sigma_M}(t) \to 0,\text{ as }t\to 0.
\end{equation}
Assume now that $M$ fails \eqref{eq:def-delta-2-limsup}, that is, there is a sequence $t_k\searrow0$ such that
\begin{equation}
  \label{eq:M2tk}
  \frac{M(t_k)}{M(2t_k)} < \frac{1}{k}.
\end{equation}
We may, of course, assume that $2t_k < 1$ and $M(2t_k) < 1$ for all $k\in\mathbb{N}$.
Define
\begin{equation}
  \label{eq:xk:def}
  i_k:=[M(2t_k)^{-1}],\quad x^k := \sum_{i=1}^{i_k}t_ke_i.
\end{equation}
By definition,
$$
  \sigma_M(2x^k) = \sum_{i=1}^{i_k} M(2t_k) = i_kM(2t_k).
$$
But $i_k \le 1/M(2t_k)$, so $\sigma_M(2x^k)\le 1$. Also, $i_k > 1/M(2t_k) -1$ and, because $M(2t_k)\to 0$, we have $\sigma_M(2x^k)\to 1$. From \eqref{F} it follows that
$$
  \|2x^k\|\to 1,\text{ as }k\to\infty\iff \|x^k\|\to 1/2.
$$
On the other hand, from \eqref{eq:M2tk}
$$
  \sigma_M(x^k) = i_kM(t_k) \le \left(\frac{1}{M(2t_k)}+1\right)M(t_k) < \frac{1}{k}+M(t_k) \to 0.\qedhere
$$
\end{proof}
If $M$ is non-degenerate and satisfies the $\Delta_2$ condition at zero we can derive  one more important implication from \eqref{eq:diam-omega-0}.
Let $x\in\ell_M$. From \eqref{eq:sigma-finite} it follows that $\sigma_M(P_N^\bot x)\to 0$ as $N\to\infty$ and from \eqref{eq:diam-omega-0} it follows that
$$
  \lim_{N\to\infty} \|P_N^\bot x\| = 0 \iff \lim_{N\to\infty}\left\|x - \sum_{i=1}^N x_ie_i\right\| = 0,
$$
that is, $(e_n)_1^\infty$ is a basis of $\ell_M$.

\section{Main results}\label{sec:main}
 We are going to state and proof our main results concerning a very specific perturbation method for Orlicz sequence spaces with $\Delta_2$. To ease the annotation, we fix throughout this section a non-degenerate Orlicz function $M$ satisfying $\Delta_2$ condition at zero, and we  denote $\sigma(\cdot):=\sigma_M(\cdot)$. We also set $X:=\ell_M$. For the same reason of simplifying the annotation, let
 \begin{equation}
    \label{eq:nu}
    \nu(t) := \sup\{ \sigma(x):\ \|x\|\le t\} < \infty,
 \end{equation}
 see \eqref{eq:sigma-finite}, and
 \begin{equation}
    \label{eq:fi}
    \varphi(t) := \mathrm{diam}\, \Omega_\sigma(t),\quad\text{ so } \varphi(t)\to 0,\text{ as }t\to 0,
 \end{equation}
 see \eqref{eq:diam-omega-0}.

For $a\in \ell_\infty$ we define the real function $g_a$ on $X$ by
\[
    g_a(x):=\sum _{n=1}^\infty a_n M(|x_n|).
\]
In this way, $\sigma = g_{(1,1,1,\ldots)}$. Recall that $ \|a\|_\infty = \sup_n |a_n|$ for $a\in \ell_\infty$.
\begin{lem}
    \label{lem:Lips}
    The function $g_a$ is bounded and Lipschitz on $KB_X$ for each $K > 0$. More precisely,
    \begin{equation}
        \label{eq:g-a-bound}
        \sup_{KB_X}|g_a|\le \|a\|_\infty\nu(K),
    \end{equation}
    and
    \begin{equation}
        \label{eq:g-a-Lipsch}
        |g_a(x)-g_a(y)| \le 2\|a\|_\infty \nu(K+1)\|x-y\|,\quad\forall x,y\in KB_X.
    \end{equation}
\end{lem}
\begin{proof}
    Assume first that $a\in\ell_\infty^+$, that is $a_n\ge0$ for all $n\in\mathbb{N}$. Then $g_a\ge 0$ and
    $$
        g_a(x) = \sum _{n=1}^\infty a_n M(|x_n|) \le \|a\|_\infty\sum _{n=1}^\infty M(|x_n|) = \|a\|_\infty\sigma(x),
    $$
    so \eqref{eq:g-a-bound} is fulfilled. Since $g_a$ is convex, it has directional derivatives $g_a'(x;h)$. If $\|x\| \le 1$ and $\|h\|=1$ then by convexity $g_a(x+h) \ge g_a(x) + g_a'(x;h)$. But $x+h\in (K+1)B_X$, so $g_a(x+h)\le \|a\|_\infty\nu(K+1)$, and $g_a(x)\ge 0$; thus $g_a'(x;h) \le \|a\|_\infty\nu(K+1)$ for all $x\in KB_X$ and all $h$ of norm $1$, meaning that $g_a$ is Lipschitz on $KB_X$ with Lipschitz constant $\|a\|_\infty\nu(K+1)$.

    Now, for arbitrary $a\in\ell_\infty^+$ we use the canonical representation $a = a_+ - a_-$, where $(a_+)_n = \max\{a_n,0\}$, so $(a_-)_n = -\min\{a_n,0\}$. Accordingly, $g_a = g_{a_+} - g_{a_-}$ and, since $a_\pm\in \ell_\infty^+$ and $\|a\|_\infty = \max\{\|a_+\|,\|a_-\|\}$, the claim follows.
\end{proof}
Set
\[
    \mathcal{P}:= \left\{ g_a\ :\  a\in  \ell_\infty^+ \right\},
\]
the latter being the cone of positive bounded sequences. We consider $\mathrm{span}\, \mathcal{P}$, that is, the set of all functions $g_a$. Let $\|g\|_{\cal P}=\|a\|_\infty$.

\begin{lem}\label{petsp}
The so defined  $(\cal{P},\|\cdot \|_\mathcal{P})$ is a perturbation space on each non-empty, closed and bounded subset of $X$.
\end{lem}

\begin{proof}
    It is clear that it is enough to check the axioms of Definition~\ref{pert_sp} for arbitrary ball. Thus, let $S\subset KB_X$ for some $K>0$. That the functions from $\mathrm{span}\,\mathcal{P}$ are continuous on $S$ is checked in Lemma~\ref{lem:Lips}. The axiom $\mathrm{(\i)}$ is clear and $\mathrm{(\i\i)}$ follows from \eqref{eq:g-a-bound}.

    To check $\mathrm{(\i\i\i)}$, let $\varepsilon > 0$ be arbitrary.  Due to \eqref{eq:fi} there is $\delta > 0$ such that
    \begin{equation}
        \label{eq:iii:delta}
        \varphi\left(\frac{3\delta}{\varepsilon}\right) < \varepsilon.
    \end{equation}
    Let $\theta\in(0,\varepsilon)$ be such that
    \begin{equation}
        \label{eq:iii:theta}
        \theta\nu(K) < \delta/2.
    \end{equation}
    Fix an arbitrary $x\in S$. Hence, $\|x\| \le K$.

    Let  $N\in\mathbb{N}$ be such that
    \begin{equation}
        \label{eq:iii:N}
        \varepsilon\sigma(P_N^\bot x) < \delta/2.
    \end{equation}
    Define
    $$
        \bar a := \theta\sum_{i=1}^N e_i + \varepsilon\sum_{i=N+1}^\infty e_i,
    $$
    and
    $$
        g(y) := g_{\bar a} (y) = \theta \sigma(P_Ny) + \varepsilon\sigma(P_N^\bot y).
    $$
    We will check that this $g$ satisfies \eqref{eq:def:perturb:c} for the fixed $x$.

    Obviously, $\|g\|_\mathcal{P} = \|\bar a\|_\infty = \varepsilon$.

    Also, $\sigma (P_Nx) \le \sigma(x)\le \nu(K)$ and $\theta\sigma (P_Nx) < \delta/2$ by \eqref{eq:iii:theta}, so using \eqref{eq:iii:N}, we get $g(x) < \delta$. Since $g(0) = 0$ and, therefore, $\min_S g = 0$, we get $x\in\Omega_g^S(\delta)$.

    Let $y\in \Omega_g^S(3\delta)$, that is, $\|y\|\le K$ and  $g(y)\le 3\delta$. In particular, $\varepsilon\sigma(P_N^\bot y) \le 3\delta \iff P_N^\bot y \in \Omega_\sigma (3\delta/\varepsilon)$. From  \eqref{eq:iii:delta}, \eqref{eq:fi} and $0\in \Omega_\sigma (3\delta/\varepsilon)$ it follows that $\|P_N^\bot y\|<\varepsilon$. Since $P_Ny\in P_NS$ and, of course, $y= P_Ny + P_N^\bot y$, it follows that the distance of $y$ to the compact set $P_NS$ is less than $\varepsilon$.

    Since $y\in \Omega_g^S(3\delta)$ was arbitrary, we get $\alpha(\Omega_g^S(3\delta)) \le \varepsilon$ and \eqref{eq:def:perturb:c} is verified.
\end{proof}

\begin{thm}\label{thm:perturb-general}
    Let $M$ be a non-degenerate Orlicz function satisfying $\Delta_2$ condition at zero and let $f:\ell_M\to\mathbb{R}\cup\{+\infty\}$ be proper lower semicontinuous and bounded below function. Then for each $\varepsilon>0$ there is $a\in\ell_\infty$ such that $\|a\|_\infty< \varepsilon$ and
    the function $f+g_a$
    attains its minimum on $\ell_M$.
\end{thm}
\begin{proof}
    Fix $\varepsilon\in(0,1)$.     Assume first that
    \begin{equation}
        \label{eq:f-to-infty}
        \lim_{\|x\|\to\infty} f(x) = \infty.
    \end{equation}
    Fix $x_0\in\mathrm{dom}\,f$ and $K > \|x_0\|$ be so big that
    $$
        \|x\|\ge K\Rightarrow f(x) > f(x_0) + \sigma(x_0).
    $$
    Set $S:=KB_X$. From Lemma~\ref{petsp} and Theorem~\ref{min} there is $a\in\ell_\infty^+$ such that $\|a\|_\infty < \varepsilon$ and $f+g_a$ attains its minimum on $S$. Let $\bar x\in S$ be such that
    $$
        f(x) + g_a(x) \ge f(\bar x) + g_a(\bar x),\quad\forall x\in S.
    $$
    But for $x\not\in S$, that is, $\|x\| > K$, we have $f(x) + g_a(x) \ge f(x) > f(x_0) + \sigma(x_0)$. Since $a_n\le 1$ for all $n\in\mathbb{N}$, we have $\sigma \ge g_a$, so $f(x) + g_a(x) \ge f(x_0) + g_a(x_0) \ge f(\bar x) + g_a(\bar x)$, because $x_0\in S$. That is, $\bar x$ is a global minimum of $f+g_a$.

    Now, if $f$ is only bounded below, fix $\theta \in (0,\varepsilon/2)$ and consider
    $$
        f_1(x) := f(x) + \theta \sigma(x).
    $$
    From \eqref{F} it follows that $\sigma(x)\to\infty$ as $\|x\|\to\infty$, so $f_1$ satisfies \eqref{eq:f-to-infty} and from the first part of this proof there is $a'\in\ell_\infty^+$ such that $\|a'\|_\infty<\theta$ and $f_1 + g_{a'}$ attains its minimum. Then $a:=(\theta+a_n')_{n=1}^\infty$ satisfies the claim.
\end{proof}
We are now ready to present our new tool for no-renorming results.
\begin{thm}\label{thm:perturb-strict}
    Let $M$ be a non-degenerate Orlicz function satisfying $\Delta_2$ condition at zero and let $f:\ell_M\to\mathbb{R}\cup\{+\infty\}$ be proper, lower semicontinuous and bounded below function with bounded domain.

    Then for any $0<\delta<\varepsilon$ there is $a\in\ell_\infty$ such that
    \begin{equation}
        \label{eq:a-bounds}
        \delta \le a_n \le \varepsilon,\quad\forall n\in\mathbb{N}
    \end{equation}and
    the function $f-g_a$
    attains its minimum on $\ell_M$.
\end{thm}
\begin{proof}
    Consider
    $$
        f_1(x) := f(x) - \varepsilon\sigma(x).
    $$
    Since $\sigma(x)$ is bounded on the bounded set $\mathrm{dom}\, f$, see \eqref{eq:sigma-finite}, the function $f_1$ is bounded below. From Theorem~\ref{thm:perturb-general} there is $a'\in\ell_\infty^+$ such that $\|a'\|_\infty < \varepsilon-\delta$ and $f_1+g_{a'}$ attains its minimum. Set
    $$
        a:=(\varepsilon - a_n')_{n=1}^\infty.
    $$
    Obviously, $a$ satisfies \eqref{eq:a-bounds} and
    $$
        f - g_a = f_1 + g_{a'}.\qedhere
    $$
\end{proof}
Note that a variant of Stegall Variational Principle for an Orlicz space such that $M$ satisfies $\Delta_2$ condition at zero (it is well known that such space is dentable, or, which is equivalent, has Radon-Nikodym property, see \cite{Phelps}) follows almost immediately from Theorem~\ref{thm:perturb-strict}. Indeed, if $\bar x$ is the point of minimum of $f-g_a$ then, assuming that $f(\bar x) = g_a(\bar x)$, we have $f\ge g_a$ and so a subdifferential of $g_a$ at $\bar x$, say $p$, will be a supporting functional to $f$, that is, $f(x)\ge f(\bar x) + p(x-\bar x)$. From \eqref{eq:g-a-Lipsch} it follows that $\|p\| < 2 \varepsilon \nu(\sup\|\mathrm{dom}\, f\| +2)$, which can be made arbitrary small.

Of course, we will not get that $\bar x$ is a strong minimum of $f-p$, but even this weaker version of Stegall Variational Principle is still equivalent to the dentability of the space, see \cite{fabine}. If one wants to get directly a strong minimum, one can regularise in advance $M$ to be strictly convex, see e.g.~\cite{LT1}.

The important point here is that the $\Delta_2$ assumption can not be just dropped in Theorems~\ref{thm:perturb-general}~and~\ref{thm:perturb-strict}.

\section{Applications}\label{sec:appl}

In this section we will demonstrate on several examples how Theorem~\ref{thm:perturb-strict} can be applied. We start with the simplest of our cases, see e.g. \cite[Theorem~II.5.3]{DGZ}. Note that most part of the proof is just checking  that the canonical norm of $\ell_1$ is nowhere Fr\'echet differentiable.

\begin{prop}\label{prop:l1}
On $\ell_1$ there is no non-trivial Fr\'echet differentiable bump.
\end{prop}

\begin{proof}
    Suppose $b$ is such a bump. Consider
    \begin{equation}
        \label{eq:b-to-f}
        f(x) := \begin{cases}
            b^{-2}(x),\quad b(x)\neq 0,\\
            \infty,\quad\text{otherwise}.
        \end{cases}
    \end{equation}
    It is clear that $f$ is proper, lower semicontinuous function which is positive and  Fr\'echet differentiable at each point of $\mathrm{dom}\, f$ and the latter is bounded set.

    Applying Theorem~\ref{thm:perturb-strict} with $\delta=1$ and $\varepsilon=2$  we get
    $a\in \ell_\infty^+$ such that $1\le a_n \le 2$ for all $n\in\mathbb{N}$, and
    $\bar x\in \mathrm{dom}\,f$ such that $f-g_a$ attains its minimum at $\bar x$. In particular
    \begin{equation}
        \label{eq:f-g-bar}
        f(x) - f(\bar x) \ge g_a(x) - g_a(\bar x),
    \end{equation}
    for all $x\in \ell_1$.
    In particular, for any $h\in \ell_1$, $h\neq 0$,
    $$
    \frac{f(\bar x+h)+f(\bar x-h)-2f(\bar x)}{\|h\|} \ge \frac{g_a(\bar x+h)+g_a(\bar x-h)-2g_a(\bar x)}{\|h\|}.
    $$
    Letting $h$ to $0$, the left hand side of the above tends to zero, because $f$ is Fr\'echet differentiable at $\bar x$, whilst the right hand side is always positive, because $g_a$ is convex. Therefore,
    \begin{equation}\label{cdc}
        \lim_{\|h\|\rightarrow 0}\frac{g_a(\bar x+h)+g_a(\bar x-h)-2g_a(\bar x)}{\|h\|} = 0.
    \end{equation}
    We get a contradiction by checking that
    $$
    \limsup _{\|h\|\rightarrow 0}\frac{g_a(\bar x+h)+g_a(\bar x-h)-2g_a(\bar x)}{\|h\|} \ge 2.
    $$
    Indeed, fix $t>0$. For all $n$ large enough $|\bar x_n| < t$, so $|\bar x_n+t| + |\bar x_n-t| -2 |\bar x_n| = 2(t-|\bar x_n|)$, and $g_a(\bar x+te_n)+g_a(\bar x-te_n)-2g_a(\bar x) = a_n(|\bar x_n+t| + |\bar x_n-t| -2 |\bar x_n|) \ge 2(t-|\bar x_n|)$. That is, $\sup_n (g_a(\bar x+te_n)+g_a(\bar x-te_n)-2g_a(\bar x)) \ge 2t $.
\end{proof}

\begin{prop}
    \label{prop:lp}
    Let $M$ be non-degenerate Orlicz function satisfying $\Delta_2$ condition at zero, and such that for some $p\in(1,2]$
    \begin{equation}
        \label{eq:Mtp}
        \limsup_{t\searrow 0}\frac{M(t)}{t^p} = \infty.
    \end{equation}
    Then on $\ell_M$ there is no nontrivial Fr\'echet differentiable bump with estimate
    \begin{equation}
        \label{eq:b-p-estimate}
        b(x+h) = b(x) + \langle b'(x),h\rangle + O(\|h\| ^p)
    \end{equation}
    for all $x\in\ell_M$.
\end{prop}
\begin{proof}
    As in the proof of Proposition~\ref{prop:l1} we define $f$ by~\eqref{eq:b-to-f}. It is immediate to check that \eqref{eq:b-p-estimate} implies similar estimate for $f$. Consequently,
    \begin{equation}
        \label{eq:f-p-estim}
        \limsup_{\|h\|\to0}\frac{f(x+h)+f(x-h)-2f(x)}{\|h\|^p} < \infty,\quad\forall x\in \ell_M.
    \end{equation}
    Further we get $g_a$ with $a_n\ge1$ for all $n\in\mathbb{N}$ such that \eqref{eq:f-g-bar} is fulfilled. Then from \eqref{eq:f-p-estim},
    \begin{equation}\label{eq:g-p-est}
        \limsup_{\|h\|\rightarrow 0}\frac{g_a(\bar x+h)+g_a(\bar x-h)-2g_a(\bar x)}{\|h\|^p} < \infty
    \end{equation}
    at some point $\bar x\in\ell_M$.

    On the other hand, let $t_k\to0$ be such that $M(t_k)> kt_k^p$, see \eqref{eq:Mtp}. Then
    \begin{multline*}
        \sup_{\|h\|=t_k} \frac{g_a(\bar x+h)+g_a(\bar x-h)-2g_a(\bar x)}{\|h\|^p} \ge\\
        \limsup_{n\to\infty} a_n\frac{M(|\bar x_n +t_k|)+M(|\bar x_n -t_k|)-2M(|\bar x_n|)}{t_k^p}\\
        \ge 2\frac{M(t_k)}{t_k^p} > 2k,
    \end{multline*}
 which yields a   contradiction.
\end{proof}

\begin{prop}\label{prop:M''}
If $M$ is a non-degenerate twice  differentiable Orlicz function satisfying $\Delta_2$ condition at $0$, and such that $\ds \lim _{t\searrow 0}M''(t)=\infty $, then on $\ell_M$ there is no twice G\^ateaux differentiable bump.
\end{prop}

\begin{proof}
    Again assume the contrary and let $f$, $g_a$ and $\bar x$ be like in \eqref{eq:b-to-f} and \eqref{eq:f-g-bar}. Then $f$ is twice Gat\^eaux differentiable on $\bar x$. Let $K:= \|f''(\bar x)\|$. In particular, the second directional derivatives of $f$ at $\bar x$ in the co-ordinate directions
    \begin{equation}
        \label{eq:co-K}
        f''_{e_n^2}(\bar x) \le K,\quad\forall n\in\mathbb{N}.
    \end{equation}
    Since $f''_{e_n^2}(\bar x) = \lim_{t\to 0} (f(\bar x +t) + f(\bar x-t) -2f(\bar x))/t^2$,  from \eqref{eq:f-g-bar} and \eqref{eq:co-K} it follows that for all $n\in\mathbb{N}$
    \begin{eqnarray*}
        K&\ge&\limsup_{t\to 0}\frac{g_a(\bar x + te_n)+g_a(\bar x - te_n)-2g_a(\bar x)}{t^2}\\
        &\ge& \limsup_{t\to 0}\frac{M(|\bar x_n + t|)+M(\bar x_n - t|)-2M(|\bar x_n|)}{t^2}.
    \end{eqnarray*}
    If $\bar x_n=0$ the latter equals $2M(t)/t^2$. However, $\limsup_{t\to0}M(t)/t^2=\infty$. Indeed, if $c:=\limsup_{t\to0}M'(t)>0$ then by convexity $M(t)\ge ct$. Otherwise by L'H\^opitale's rule, $\lim_{t\to0}2M(t)/t^2 = \lim_{t\to0}M''(t) = \infty$, by assumption. This means that $\bar x_n\neq0$ for all $n\in\mathbb{N}$. But then the last limit superior above is simply $M''(|\bar x_n|)$. So,
    $$
        K\ge M''(|\bar x_n|),\quad \forall n\in\mathbb{N},
    $$
    which yields a contradiction to $\bar x_n\to 0$.
\end{proof}

\medskip

\textbf{Acknowledgements.} The authors express their sincere gratitude to Dr. Milen Ivanov for the fruitful discussions during the research and for  his help and constant support.

\end{document}